\documentclass{article}
\usepackage{enumerate}
\usepackage{graphicx}
\usepackage{hyperref}
\usepackage[english]{babel}
\usepackage[utf8]{inputenc}
\usepackage{amsmath}
\usepackage{amssymb}
\usepackage{amsthm}
\usepackage{fullpage}

\newtheorem{thm}{Theorem}[section]
\newtheorem{lemma}[thm]{Lemma}

\newtheorem{observation}[thm]{Observation}

\newcommand{\brac}[1]{\left\lbrace #1 \right\rbrace}
\newcommand{\card}[1]{\left| #1 \right|}

\newcommand{\st}{\colon\ }

\date{}

\title{Guarding isometric subgraphs and \\ Cops and Robber in planar graphs}

\author{Sebasti\'an Gonz\'alez Hermosillo de la Maza\\
{\tt sghm511@gmail.com} \\
Department of Mathematics \\
Simon Fraser University\\
8888 University Drive\\
Burnaby, BC, Canada
\and
Bojan Mohar\thanks{Supported in part by the NSERC Discovery Grant R611450 (Canada) and by the Research Project J1-8130 of ARRS (Slovenia).}\\
{\tt mohar@sfu.ca}\\
Department of Mathematics \\
Simon Fraser University\\
8888 University Drive\\
Burnaby, BC, Canada
}

\begin{document}

\maketitle

\begin{abstract}
In the game of Cops and Robbers, one of the most useful results is that
an isometric path in a graph can be guarded by one cop. In this paper, we introduce
the concept of wide shadow in a subgraph, and use it to characterize all 1-guardable graphs.
As an application, we show that 3 cops can
capture a robber in any planar graph with the added restriction that at most two
cops can move simultaneously, proving a conjecture of Yang and strengthening a
classical result of Aigner and Fromme.
\end{abstract}


\section{Introduction}

The game of \emph{cops and robbers} in the setting of one cop and one robber was introduced by Quillot \cite{QuilliotD1978} and studied independently by Nowakowski and Winkler \cite{NowakowskiDM1983}. The game was generalized by Aigner and Fromme in \cite{AignerDAM1984} to allow multiple cops.
The game is played on a connected graph $G$ by two players; one controls the moves of $k$ cops (where $k$ is a fixed positive integer), while the other one controls the robber. At the beginning, each cop chooses a vertex as his starting position, and after that the robber chooses his initial position. The players then move in alternate turns, and in each turn a player might stay at his current position, or move to an adjacent vertex. The \emph{cops win} if one of the cops eventually occupies the same vertex as the robber, a situation we will refer to as \emph{capturing the robber}, while the \emph{robber wins} if he is able to indefinitely prevent this from happening.

Aigner and Fromme \cite{AignerDAM1984} defined the \emph{cop-number} of a graph $G$, which we will denote by $c(G)$, to be the smallest integer $k$ such that $k$ cops can guarantee robber's capture on $G$ regardless of his strategy. Since for any graph $G$ we have $c(G) \leq \card{V(G)}$, the cop-number is well defined for every finite graph. Moreover, for any graph $G$ we have $c(G) \leq \gamma(G)$, where $\gamma(G)$ denotes the domination number of $G$.
A graph $G$ with $c(G)=1$, i.e., a single cop has a winning strategy, is called a \emph{cop-win graph}.

A lot of research has been done studying connections between topological properties of graphs and their cop-number. See \cite{BonatoA2017} for a survey and \cite{BonatoAMS2011} for more insight. The first result of this type is due to Aigner and Fromme \cite{AignerDAM1984}, who showed that $c(G) \leq 3$ for any planar graph $G$. The main tool they used was a lemma showing that a single cop can guard any isometric path in the graph (see Section \ref{sect:wide shadow} for details). This result is widely used and is the main reason that the game is so interesting. Thus, it is natural to ask if cops can guard other kinds of subgraphs. Lu and Wang \cite{LuWang18} found an infinite family of such graphs.
This paper goes much further and gives a characterization of those isometric subgraphs that can be guarded by using a single cop. See Theorem \ref{Hellyguard}. As it turns out, the class of guardable subgraphs consists of precisely those graphs that possess the Helly property with respect to vertex neighborhoods. This class is tightly related to absolute retracts (see \cite{BandeltJCTB1991,HellOLS2004}).

The proof of Aigner and Fromme \cite{AignerDAM1984} that isometric paths can be guarded is based on a notion of what we call here a \emph{shadow}.
The vertex, where the cop is placed to guard the path is referred to as the \emph{shadow of the robber}. This idea has been generalised in the context of graph homomorphisms in terms of absolute retracts, see Theorems \ref{1guardable} and \ref{bandelt}. We introduce a novel concept of a \emph{wide shadow}. This notion provides a powerful new tool in the theory of cop numbers. We use it to prove a conjecture of Yang \cite{YangCCC2018} concerning the following variant of the cops and robber game.
In \cite{OffnerAJC2014}, Offner and Ojakian introduced a variation where only one cop is allowed to move at each turn, and they referred to it as the \emph{one-active-cop} game. Shortly after, the same variation was introduced with different names, like \emph{lazy cops and robber} \cite{BalCPC2015}, and the \emph{one-cop-moves game} \cite{GaoCOA2017}. In this paper we follow the naming in \cite{GaoCOA2017}, and define the \emph{$k$-cops-move} number of a graph, $c_k(G)$, as the smallest integer such that $c_k(G)$ cops guarantee the robber's capture in $G$ with the restriction that at most $k$ cops can change their position at each turn.

Sullivan \emph{et al.}~\cite{SullivanA2016} showed that every graph $G$ on at most eight vertices satisfies $c_1(G) \leq 2$, and that there is a unique graph on nine vertices with $c_1(G)=3$. For many classes of planar graphs we have $c_1(G) = c(G)$, so they posed the question of whether there exists a planar graph $G$ with $c_1(G) \geq 4$. Gao and Yang constructed a planar graph for which $c_1(G) \geq 4$, and Yang \cite{YangCCC2018} conjectured that $c_2(G) \leq 3$ for any planar graph. Our second main result, Theorem \ref{twocopmove}, settles Yang's conjecture in the affirmative. Our concept of wide shadow is the main tool used in the proof.


\subsection{Basic notation}

All graphs in this paper are connected unless stated otherwise. Let $G$ be a graph and $X,Y \subseteq V(G)$. We will use $d_G(X,Y)$ or $d(X,Y)$ to denote the length of a shortest $(X,Y)$-path in $G$, omitting $G$ when there is no confusion. In the case $X$ has one vertex, we will write $d(x,Y)$ instead of $d(\brac{x},Y)$. The analogous will be done when $Y$ or both sets consist of a single vertex.

A subgraph $H$ of a graph $G$ is an \emph{isometric subgraph} of $G$ if for any $x,y\in V(H)$, $d_H(x,y)=d_G(x,y)$.
If $X\subseteq V(G)$, we denote by $G[X]$ the subgraph of $G$ induced by the vertex-set $X$.

Let $G$ be a graph, $H$ a subgraph of $G$, $x \in V(H)$ and $k$ a non-negative integer. We set $N_H^k[x] = \brac{y \in V(H)\st d_H(x,y) \leq k}$, and $N^k_H(x) = N^k_H[x] \setminus \brac{x}$. We will omit the superscript $k$ when $k=1$, and we will use $N^k(x)$ instead of $N^k_H(x)$ when it causes no confusion (usually when $H=G$). By a \emph{walk} in $G$ we mean a sequence $v_1v_2\dots v_s$ of vertices of $G$ such that $v_iv_{i+1}\in E(G)$ for each $i=1,\dots,s-1$. If $P = v_1v_2\cdots v_s$ and $Q = u_1u_2\cdots u_t$ are walks and $v_i = u_j = x$ for unique values of $i \in [s]$ and $j \in [t]$, then the walk $v_1v_2\cdots v_{i-1}xu_{j+1}\cdots u_t$ will be denoted by $PxQ$.

\section{Wide shadow and Helly graphs}\label{anew}
\label{sect:wide shadow}

Let $G$ be a graph and $H$ an isometric subgraph of $G$. For $k\ge1$, we say that $H$ is \emph{$k$-guardable} if, after finitely many moves, $k$ cops can move on vertices of $H$ and from that point on stay in $H$ such that, if the robber ever enters $H$, then he will be captured by one of these cops in the next turn.

One of the most important tools to study the game of cops and robbers is the following result of Aigner and Fromme \cite{AignerDAM1984} claiming that isometric paths can be guarded by a single cop.

\begin{thm}[\cite{AignerDAM1984}]\label{AignerandFromme}
	Every isometric path is $1$-guardable.
\end{thm}

A necessary condition for being able to guard a subgraph $H$ is that for every position $v$ of the robber, there is a vertex of $H$ such that the cop being at this vertex is able to reach every vertex in $H$ at least as quickly as the robber. Such a vertex is called a \emph{shadow} of the robber.
To define a shadow on an isometric path $P$, pick an endpoint $s$ of the path. For any vertex $v \in V(G)$, let the \emph{shadow} of $v$ with respect to $s$ be the vertex $y \in V(P)$ such that $d(s,y) = d(s,v)$ when $d(s,v) \leq \card{E(P)}$, and otherwise let it be the endpoint of $P$ different from $s$.

More generally, let $H$ be an isometric subgraph of $G$.  For $v \in V(G)$ and $x \in V(H)$, let
$$\gamma(H,x,v) = \brac{y \in V(H)\mid d(y,x) \leq d(v,x)}.$$
Let us note that $\gamma(H,x,v)$ is the set of vertices in $H$ from which a cop can protect $x$ when the robber is at $v$. It is especially nice if a cop is at a vertex $y$ from which he can protect all vertices in $H$. Motivated by this interpretation,
we define the \emph{wide shadow} of $v$ on $H$ to be the set
\begin{equation}\label{eq:wideshadow}
  S_H(v) = \bigcap_{x \in V(H)} \gamma(H,x,v).
\end{equation}

The following restatement of the definition of the wide shadow shows more clearly that the wide shadow is actually composed of all possible shadows.

\begin{observation}
  A vertex $y\in V(H)$ belongs to the wide shadow $S_H(v)$ if and only if for every $x\in V(H)$, $d(y,x)\le d(v,x)$.
\end{observation}

\begin{figure}[h]
\begin{center}
		\includegraphics[width=0.65\textwidth]{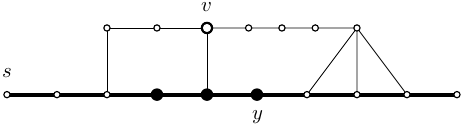}
\end{center}
\caption{The path $P$ consisting of thick edges is isometric in $G$. The vertex $y$ is the shadow of the vertex $v$ with respect to $s$. The three black vertices form the wide shadow of $v$.}\label{wideshadowfig}
\end{figure}

Figure \ref{wideshadowfig} shows an isometric path $H$ (depicted as the path using thick edges) and shows the shadow and the wide shadow of a vertex outside $P$.

It is natural to ask whether one can generalize Theorem \ref{AignerandFromme} to include other classes of graphs in addition to all paths. Quite recently, Lu and Wang \cite{LuWang18} found an infinite family of such graphs, called \emph{vertebrate graphs}. Their definition is complicated, and so is the proof that they can be guarded by a single cop. In this paper we go further by providing complete characterization of all such graphs. As shown in \cite{Seb_thesis}, our characterization yields a much simpler proof that vertebrate graphs are 1-guardable.

Let us go back to the definition (\ref{eq:wideshadow}) of the wide shadow. The condition that the wide shadow is non-empty, which is a necessary condition for having a shadow in $H$, is reminiscent to the Helly property from discrete geometry (see, e.g.,~\cite{DaGrKl63}). This leads to the following definition.

A family of sets $\mathcal{S}$ has the \emph{Helly property} if for every $\mathcal{T} \subseteq \mathcal{S}$ we have the following property: if $X_1 \cap X_2 \neq \varnothing$ for every $X_1, X_2 \in \mathcal{T}$, then $\bigcap \mathcal{T} \neq \varnothing$. A graph is a \emph{Helly graph} if the family $\brac{N^k[u]\st u \in V(G), k \geq 0}$ has the Helly property. It is well known that all trees are Helly graphs.

Helly graphs have the basic property required in order to be $1$-guardable in a graph $G$ whenever they appear as an isometric subgraph of $G$.
In general, the wide shadow of a vertex on an isometric subgraph $H$ may be empty, but this is not the case when $H$ is a Helly graph.

\begin{lemma}\label{wideshadow}
	Let $G$ be a graph and $H$ an isometric subgraph of $G$. If $H$ is a Helly graph, then the following statements hold.
	\begin{enumerate}
		\item[$(i)$]  For every $v \in V(G)$, $S_H(v) \neq \varnothing$.
		\item[$(ii)$] For every $uv \in E(G)$, and every $y \in S_H(u)$, we have $d(y,S_H(v)) \leq 1$.
	\end{enumerate}	
\end{lemma}

\begin{proof}
	Let $x,y \in V(H)$. Since $H$ is isometric, we have $d_H(x,y) = d_G(x,y)$. By the triangle inequality we have $d_H(x,y) \leq d_G(x,v) + d_G(v,y)$, which implies $\gamma(H,x,v) \cap \gamma(H,y,v) \neq \varnothing$. Since $H$ is a Helly graph, we know that its ball-hypergraph $\mathcal{H}$ has the Helly property. In particular, the family $\{\gamma(H,x,v) \mid x\in V(H)\}$ has a nonempty intersection, i.e., $S_H(v) \neq \varnothing$. This proves $(i)$.
	
	Now, let $uv \in E(G)$. Notice that for every $x \in V(H)$, we have $d(x,u)-1 \leq d(x,v) \leq d(x,u)+1$. This implies $\gamma(H,x,v) \cap N_H[y] \neq \varnothing$ for every $x \in V(H)$ and $y \in S_H(u)$. Since $H$ is Helly, the intersecting family $\{\gamma(H,x,v) \mid x\in V(H)\}\cup \{N_H[y]\}$ has a nonempty intersection, i.e., we have $S_H(v) \cap N_H[y] \neq \varnothing$ for every $y \in S_H(u)$. It follows that for each $y \in S_H(u)$ either $y \in S_H(v)$ or there exists $z\in N_H(y) \cap S_H(v)$. Hence $d(y,S_H(v)) \leq 1$, completing the proof of $(ii)$.
\end{proof}

We say that a vertex $x \in V(G)$ is a \emph{corner} of $G$ if there exists $y \in N(x)$ such that $N[x] \subseteq N[y]$. It is easy to see that $c(G)=c(G-x)$ if $x$ is a corner. Thus, we can perform a \emph{corner-deletion} by removing a corner from $G$ without altering the cop number of the graph. A graph is said to be \emph{dismantlable} if it can be reduced to a single vertex by a sequence of corner-deletions. Thus, dismantlable graphs are cop-win. It was shown by Quilliot \cite{Quilliot1983} (and rediscovered later in \cite{BandeltJCTB1991}) that Helly graphs are dismantlable. We include this result as Lemma \ref{lem:Hellyiscopwin} below, and provide a proof for the sake of completeness.

\begin{lemma}\label{lem:Hellyiscopwin}
	If $G$ is a Helly graph, then $G$ is dismantlable and hence cop-win.
\end{lemma}

\begin{proof}
Let $u$ and $v$ be vertices of $G$ with $d(u,v) = d$, where $d$ is the diameter of $G$. Let $T_u = N^{d-1}(u)$, and for $x \in N[v]$, let $T_x = N[x]$. Clearly $T_x \cap T_y \neq \varnothing$ for all $x,y \in N[v]$, and since $d(u,v) = d$, we get $T_u \cap T_x \neq \varnothing$. Since $G$ is a Helly graph, there exists $z \in T_u \cap \left(\bigcap_{x \in N[v]} T_x\right)$. Since $z \neq v$ and $z \in T_x$ for every $x \in N[v]$, we have $N[v] \subseteq N[z]$, so $G$ has a corner.

Notice that if $G$ is a Helly graph and $x$ is a corner of $G$, then $G - x$ is also a Helly graph. The result now follows by induction on $\card{V(G)}$.
\end{proof}

We are ready to present our first main result.

\begin{thm}\label{Hellyguard}
A graph $H$ is $1$-guardable in every graph $G$ that contains $H$ as an isometric subgraph if and only if $H$ is a Helly graph.
\end{thm}

\begin{proof}
Let $G$ be a graph and $H$ an isometric subgraph which is a Helly graph. It follows from Lemmas \ref{wideshadow} and \ref{lem:Hellyiscopwin} that $H$ is 1-guardable: we can capture the wide shadow of the robber by Lemma \ref{lem:Hellyiscopwin}, and stay in it by Lemma \ref{wideshadow}. This proves the ``only if'' part of the theorem.

For the ``if'' part, suppose $H$ is 1-guardable, but not a Helly graph. Since $H$ is an isometric subgraph of itself, being 1-guardable means in particular that it is cop-win. Not being Helly means that there exist vertices $v_1,\dots, v_k$ and positive integers $d_1,\dots, d_k$ such that $N^{d_i}(v_i) \cap N^{d_j}(v_j) \neq \varnothing$ for all $1\le i<j\le k$, but $\bigcap_{i=1}^k N^{d_i}(v_i) = \varnothing$. Let $G$ be the graph obtained from $H$ by adding a vertex $x$ and internally disjoint paths $\brac{P_i}_{i=1}^k$, with $P_i$ having length $d_i$, joining $x$ with each $v_i$. Note that $H$ is an isometric subgraph in $G$.

Assume that the robber is in $x$. Since $H$ is $1$-guardable, a cop has a position in $H$ and a strategy that enables him to capture the robber when he enters $H$. Let us assume that the cop is at a vertex $u\in V(H)$ from which he can follow such strategy, at that there is robber's turn. Since $\bigcap_{i=1}^k N^{d_i}(v_i) = \varnothing$, there exists a vertex $v_j$ such that $d_H(u,v_j) > d_j$. Now, if the robber moves along $P_j$, he will be able to get to vertex $v_j$ in exactly $d_j$ steps. However, since $d_H(u,v_j) > d_j$,  the cop will be at least at distance two from $v_j$ when the robber enters $v_j$, so he cannot capture the robber in the next turn. This is a contradiction, so $H$ must be a Helly graph.
\end{proof}

Let us put Theorem \ref{Hellyguard} in a historical context. Helly graphs have known relationship with the notion of absolute retracts (see \cite{HellOLS2004}), and the notion of a hole of a graph as defined in \cite{HellOLS2004}.
A \emph{hole} in a graph $G$ is a set of $k \geq 3$ vertices $v_1, \dots, v_k$ and positive integers $d_1, \dots, d_k$ such that:
\begin{itemize}
	\item $G$ contains no vertex $u$ with $d(u,v_i) \leq d_i$ for all $i \in [k]$, and
	\item for every $i,j \in [k]$, we have $d(v_i,v_j) \leq d_i + d_j$.
\end{itemize}
Let us observe that the notion of a hole was used in the proof of Theorem \ref{Hellyguard}.

\begin{thm}[\cite{HellOLS2004}]\label{bandelt}
A graph is a Helly graph if and only if it is an absolute retract.
\end{thm}

\begin{thm}[\cite{BandeltJCTB1991}]\label{1guardable}
A graph is an absolute retract if and only if it has no holes.
\end{thm}
	
These old results show that our Theorem \ref{Hellyguard} may have been at reach already 30 or more years ago. But we could not find it in the existing literature.

\subsection*{More on wide shadows}

Theorem \ref{Hellyguard} in particular implies Theorem \ref{AignerandFromme} since paths are Helly graphs. However, when $H$ is a path, there is a difference between this approach and the one usually used to guard a path: the wide shadow of the robber will in general be a set with more than one vertex of $H$ unless the robber's position satisfies some specific conditions. In Section \ref{sect:planar}, this will allow us to devise a strategy in which the cop guarding a path can rest from time to time, and will allow us to use those turns to move other cops, and thus maintain the required property that there are not too many cops active in any given turn.
	
In order to describe when the wide shadow of the robber will be more than a single vertex, we need to define an additional structure which we call a bypath, and show its connection with the wide shadow of the robber.
	
Let $G$ be a graph, $H$ a subgraph of $G$, and $P = v_1v_2 \cdots v_k$ an isometric path in $H$. A path $B = b_1b_2 \cdots b_t$ with $t\ge3$, is called a \emph{bypath} of $P$ in $H$ if $B \subseteq V(H)$,
$b_1 = v_i$, $b_t = v_j$, with $1\le i < j \le k$, $V(B) \cap V(P) = \brac{v_i,v_j}$, and the path $Pb_1Bb_tP$ is also an isometric path in $H$.
Note that in this case $j-i=t-1$.
The vertices $v_i$ and $v_j$ are called the \emph{branching vertices} of $B$, the path obtained after replacing vertices of $P$ with $B$ is called the \emph{$B$-bypath of $P$}, and we will denote it by $P_{\left\langle B \right\rangle}$.  A path $P$ is \emph{bypath-free} in $H$ if $H$ contains no bypath of $P$. Observe that if $P$ is bypath-free in $G$, then every subpath of $P$ is bypath-free in $G$. Notice that all paths of length $1$ are trivially bypath-free in $G$. In the following, we will refer to paths of length at least $2$ as non-trivial.
	
\begin{figure}[h]
	\begin{center}
		\includegraphics[width=.5\textwidth]{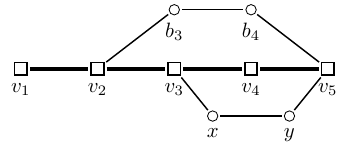}
	\end{center}
	\caption{The path $B = v_2b_3b_4v_5$ is a bypath of $P = v_1v_2v_3v_4v_5$, while $v_3xyv_5$ is not.}
\end{figure}
			
\begin{lemma}\label{bypath}
Let $G$ be a graph and $P$ a non-trivial isometric path in $G$. If $v\in V(G)\setminus V(P)$, then $\card{S_P(v)} = 1$ if and only if there exists a bypath $B$ of $P$ in $G$ such that $v \in V(B)$.
\end{lemma}

\begin{proof}
Notice that if $B$ is a bypath of $P$ with branching vertices $u$ and $w$, then $d(v,u) + d(v,w) = d(u,w)$ for every $v\in V(B)$, so $\card{\gamma(P,u,v) \cap \gamma(P,w,v)} = 1$. By Lemma \ref{wideshadow} we know that $S_P(v) \neq \varnothing$, so the fact that $S_P(v) \subseteq \gamma(P,u,v) \cap \gamma(P,w,v)$ implies that $\card{S_P(v)} = 1$.
		
For the other direction, let $v \in V(G)\setminus V(P)$ such that $S_P(v) = \brac{x}$. Since $P$ is a path and $\gamma(P,u,v)$ is a nontrivial subpath of $P$ for every $u\in V(P)$, there exist vertices $u,w \in V(P)$ such that $\gamma(P,u,v) \cap \gamma(P,w,v) = \brac{x}$. Let $y,z$ be the vertices of $P$ such that $d(y,z)$ is minimum subject to the condition that $\gamma(P,y,v) \cap \gamma(P,y,v) = \brac{x}$. Note that
\begin{equation}\label{eq:distanceonP}
  d(y,v) = d(y,x),~ d(z,v) = d(z,x), \hbox{ and } d(y,z) = d(y,x)+d(x,z).
\end{equation}
		
Let $B_y$ and $B_z$ be isometric paths from $v$ to $y$ and $z$, respectively. Notice that $V(B_y) \cap V(B_z) = \brac{v}$ (if not, (\ref{eq:distanceonP}) would contradict the assumption of $P$ being isometric), so concatenating $B_y$ and $B_z$ results in a path $B$ from $y$ to $z$. Now, let us take $P' = PyBzP$. By (\ref{eq:distanceonP}), we have $d_P(y,z) = d_{P'}(y,z)$, and it is easy to see that $P'$ is an isometric path. Hence, $B$ is a bypath of $P$ in $G$.
\end{proof}

\section{Other ways to guard a subgraph}

While the problem of $1$-guarding an isometric subgraph is well understood thanks to our Theorem \ref{Hellyguard}, we may want to extend this notion further and guard an isometric subgraph $H$ of $G$ using several cops. This poses the problem of finding properties that imply that $H$ can be $k$-guarded for some $k \geq 2$. The idea of finding a partition of $H$ into subgraphs that are $t$-guardable for $t < k$ has been explored by Clarke in \cite{ClarkeT2002}, where the problem of partitioning a graph into isometric paths was studied. It would be of interest to look at the problem of partitioning a graph into isometric Helly graphs, in particular into trees.

On the other hand, we may ask if there is a function $f(k)$ such that, if $H$ is an isometric subgraph of $G$ and $c(H) = k$, then $H$ can be $f(k)$-guarded.  Unfortunately, this fails even for $k = 1$. It is not hard to check that the subgraph $H$ in Figure \ref{hajos} cannot be guarded by a single cop despite being isometric and dismantlable, but two cops suffice. In this section we provide a family of graphs, generalizing the one in Figure \ref{hajos} which contain dismantlable isometric subgraphs that require arbitrarily many cops to be guarded.

\begin{figure}[h]
	\begin{center}
			\includegraphics[width=.2\textwidth]{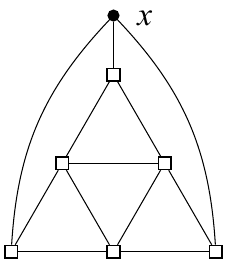}
	\end{center}
	\caption{The subgraph $H=G-x$ induced by the square vertices is dismantlable but cannot be guarded by a single cop.}\label{hajos}	
\end{figure}

Let $s\ge1$ and $t\ge3$ be integers, where $t=2m-1$ is odd. Let $T$ be a complete graph with $t$ vertices and let $\mathcal{M} = \dbinom{V(T)}{m}$ be the family of all $m$-subsets of $V(T)$. For each $X \in \mathcal{M}$, let $K_X$ be a copy of the complete graph on $s + m$ vertices. We define $H(t,s)$ as the graph obtained from the disjoint union of $T$ and all $K_X$ ($X\in \mathcal{M}$) by identifying $m$ vertices in $K_X$ with the set $X\subset V(T)$ for each $X\in \mathcal{M}$. Note that $H(t,s)$ is a chordal graph of diameter 2 and has $t+s\binom{t}{m}$ vertices.

\begin{thm}
For every positive integer $k \geq 3$ there exist a graph $G$ containing an isometric subgraph $H$ such that
$c(H) = 1$, but $k$ cops cannot guard $H$ in $G$.
\end{thm}

\begin{proof}
Let $k,s$ and $t$ be positive integers such that $t=2m-1$ is odd and $s,m > k$, and consider the graph $H = H(t,s)$.

Notice that any dominating set of $H$ must contain at least $m$ vertices since no smaller set can have a vertex in each $K_X$, and any subset of $V(T)$ with $m$ vertices dominates $H$. Also, a subset $A$ of $V(H)$ is dominating if and only if $A \cap V(K_X) \neq \varnothing$ for every $X$. We know that $c(H) = 1$ since $H$ is chordal. Also, $H$ has diameter 2, since every $X,Y \in \mathcal{M}$ have a nonempty intersection.

Let $\mathcal{A}$ be a family of all subsets $S \subseteq V(H)\setminus V(T)$ such that $\card{S \cap K_X} = 1$ for  each $X \in \mathcal{M}$.
Let $G$ be the graph obtained from $H$ by adding a vertex $v_S$ adjacent to every vertex in $S$, for each $S \in \mathcal{A}$. Since $H$ has diameter two, $H$ is an isometric subgraph of $G$. Suppose now that we have $k$ cops in $H$. Since $m > k$, there exists $X_0 \in \mathcal{M}$ such that there is no cop in $K_{X_0}$. Also, since $s > k$, there exists $S \in \mathcal{A}$ such that there is no cop in $N[v_S]$, so we can place the robber in $v_S$. Notice that regardless of how the cops chose to move, there exists a vertex $u$ in $N(v_S)$ that no cop can reach in one move, so the robber can move to $u$, thus entering $H$, and no cop can capture him in the following turn. Since $s>k$, after the cops' move there is a set $S'\in \mathcal{A}$ such that $v_{S'}$ is adjacent to $u$ and is not adjacent to any of the cops. Then the robber can move to $v_{S'}$. In this way the robber can evade the cops indefinitely and enter $H$ in every second step.
\end{proof}

\section{The 2-cop-move number of planar graphs}
\label{sect:planar}

Let us recall that $c_k(G)$ is the minimum value of $t$ such that $t$ cops can catch the robber even if at most $k$ cops are allowed to move in any step of the game. The purpose of this section is to prove a conjecture of Yang \cite{YangCCC2018} that $c_2(G)\leq 3$ for every connected planar graph $G$.
	
We say that an induced subgraph $H$ is \emph{$k$-leisurely-guardable} if it is $k$-guardable, and there exists a $k$-guarding strategy of $H$ using $k$ cops with the property that there is an integer $t$ and in every sequence of $t$ consecutive moves, there is a move in which at least one of these cops stays still.

\begin{lemma}\label{lax}
Let $G$ be a graph and $P$ an isometric path in $G$. If $P$ is bypath-free in $G$, then $P$ is $1$-leisurely-guardable.
\end{lemma}
	
\begin{proof}
Let $r$ denote the current position of the robber.
After a finite number of moves, we can get one cop $C$ to move to a vertex in the wide shadow of the robber in $P$. Once $C$ is in $S_P(r)$, he will stay still if he is in the changed wide shadow $S_P(r)$ after robber's turn, and will move when his position after the robber's turn is not in the wide shadow of the robber. By Lemma \ref{wideshadow}, the cop can always get back in the wide shadow of $r$ with a single move. By staying in the wide shadow of the robber, the cop can $1$-guard $P$.
		
If $P$ has length $\ell$, with $\ell \leq 2$, the cop can guard it without moving by simply staying at some vertex of $P$, so we may assume $\ell \geq 3$. Since $P$ is bypath-free in $G$, we have $\card{S_P(r)} \geq 2$, so the robber can only move at most $\ell$ consecutive times without entering $P$ and forcing the cop to move in order to stay in his wide shadow. Hence, after at most $\ell$ turns, the robber will either enter $P$ or the cop can stay still and be $1$-guarding $P$, so $P$ is $1$-leisurely-guardable.
\end{proof}

\begin{lemma}\label{bypath-free}
Let $G$ be a planar graph and $H$ a connected induced subgraph of $G$. If $v \in V(H)$, then one of the following holds:
\begin{itemize}
	\item[i.] $N[v] \supseteq V(H)$, or
	\item[ii.] there exists a non-trivial isometric path $P$ in $H$, containing $v$, which is bypath-free in $H$, or
	\item[iii.] for every $z \in V(H)$ there exist distinct neighbors $x_1,x_2$ of $v$ in $H$ and a non-neighbor $u\in V(H)$ such that $P_1=vx_1u$ and $P_2=vx_2u$ are isometric paths in $H$, and there is an induced subgraph $R\subseteq H$ that contains $z$ and $P_1\cup P_2$, such that $P_1$ is bypath-free in $R-x_2$ and $P_2$ is bypath-free in $R-x_1$.
\end{itemize}
\end{lemma}

\begin{proof}
Suppose that $N[v] \cap V(H) \neq V(H)$. Since $H$ is connected, there exists a vertex $u \in V(H)$ such that $d_H(u,v) = 2$. Let $x$ be a common neighbor of $v$ and $u$, $x \in N_H(v) \cap N_H(u)$. If $x$ is the only common neighbour, then $P=vxu$ is isometric and bypath-free in $H$.

Otherwise, $N_H(u) \cap N_H(v) = \brac{x_1, x_2, \dots, x_k}$ for some $k \geq 2$. Let $P_i = vx_iu$ for $i\in [k]$. Notice that the paths in $\mathcal{Q} = \brac{P_i\mid i \in [k]}$ are isometric, and they naturally induce a quadrangulation of the plane, where each quadrangular face has exactly two paths in $\mathcal{Q}$ as its boundary.

We may assume that the quadrangular face $F$ bounded by $P_1$ and $P_2$ contains the vertex $z$. Let $R$ be the subgraph of $H$ induced by $V(P_1\cup P_2)$ and all vertices in $F$. Note that $R$ contains $z$.
It remains to show that $P_i$ is bypath-free in $R-x_{3-i}$ (for $i=1,2$). But this is clear since $R$ contains none of the common neighbors of $v$ and $u$ except for $x_1$ and $x_2$.
\end{proof}

From now on, we assume that $G$ is a planar graphs with a fixed embedding in the plane. Then every cycle in $G$ bounds a disk. By considering an embedding of $G$ in the 2-sphere instead of the plane, we will be able to always assume that whenever we guard a cycle in $G$, we will assume the robber is in the curve's interior (by declaring the disk he is in as the interior). By keeping this in mind, given two internally disjoint paths $P$ and $Q$ with the same endpoints in $G$, we will use $R(P,Q)$ to denote the set of vertices contained in the interior of the disk bounded by $P \cup Q$.

The following lemma provides us with a way of forcing the robber to be inside a region whose boundary can be leisurely guarded.

\begin{lemma}\label{movepath}
Let $G$ be a planar graph, $u,v \in V(G)$, $P$ and $Q$ be internally disjoint $(u,v)$-paths in $G$. Suppose that
\begin{itemize}
	\item $P$ is isometric and bypath-free in $H = G[R(P,Q) \cup  V(P) \cup V(Q)]$ and
	\item $Q$ is isometric in $H_Q = G[R(P,Q) \cup V(Q)]$.
\end{itemize}
Then either $Q$ is bypath-free in $H_Q$ or there exists a $B$-bypath of $Q$ in $H_Q$ such that the following holds:
\begin{enumerate}
	\item[i.] $P$ is isometric and bypath-free in $G\left[R\left(P,Q_{\left\langle B \right\rangle}\right) \cup  V(P)\cup V\left(Q_{\left\langle B \right\rangle}\right) \right]$.
	\item[ii.] $Q$ is isometric and bypath-free in $G\left[R\left(Q,Q_{\left\langle B \right\rangle}\right) \cup  V(Q)\right]$.
	\item[iii.] $Q_{\left\langle B \right\rangle}$ is isometric in $G\left[R\left(P,Q_{\left\langle B \right\rangle}\right) \cup V\left(Q_{\left\langle B \right\rangle}\right)\right]$, and is isometric and bypath-free in 
\\
$G\left[R\left(Q,Q_{\left\langle B \right\rangle}\right) \cup  V(Q_{\left\langle B \right\rangle})\right]$.
\end{enumerate}
\end{lemma}

\begin{proof}
	Suppose that $Q$ is not bypath-free in $H_Q$ and let $B$ be a bypath of $Q$ in $H_Q$ such that $R(Q_{\left\langle B' \right\rangle},Q) \subseteq R(Q_{\left\langle B \right\rangle},Q)$ implies $B' = B$ for every $B'$-bypath of $Q$ in $H_Q$. Let us check that $B$ satisfies the desired properties.
	
	Clearly, $(i)$ follows from the fact that $R(P,Q_{\left\langle B \right\rangle}) \subseteq R(P,Q)$. Since $Q$ is isometric in $H_Q$, $Q$ and $Q_{\left\langle B \right\rangle}$ have the same length, and so are isometric in $G[R(Q,Q_{\left\langle B \right\rangle}) \cup V(Q) \cup B]$. If $Q$ has a bypath $B'$ in $G[R(Q,Q_{\left\langle B \right\rangle}) \cup V(Q)]$, then $R(Q,Q_{\left\langle B' \right\rangle}) \subseteq R\left(Q,Q_{\left\langle B \right\rangle}\right)$. Our choice of $B$ implies that $B = B'$, a contradiction. If $Q_{\left\langle B \right\rangle}$ has a bypath $B'$ in $G[R(Q,Q_{\left\langle B \right\rangle}) \cup V\left(Q_{\left\langle B \right\rangle}\right)]$,  take $Q' = (Q_{\left\langle B \right\rangle})_{\left\langle B' \right\rangle}$ so then $R(Q,Q') \subsetneq R\left(Q,Q_{\left\langle B \right\rangle}\right)$. In this case, there exists a bypath $B''$ in $H_Q$ such that $Q_{\left\langle B'' \right\rangle} = Q$, which contradicts the choice of $B$. This justifies $(ii)$ and $(iii)$.
\end{proof}

Using Lemmas \ref{lax} and \ref{movepath}, we will follow the usual strategy of Aigner and Fromme \cite{AignerDAM1984} to use two isometric paths to bound the subgraph where the robber can stay and use the third cop helping them to shrink the territory of the robber. At any point of the game, two cops will be leisurely guarding a subgraph $Z$ of $G$. Then we define the vertices of the component $H$ of $G-Z$ containing the robber as the \emph{robber's territory}. The rest of the vertices form the \emph{cop territory}. Notice that if the robber's territory is empty, that means the robber has been captured. The idea of the proof is to play the game using three cops in such way that we reduce the size of the robber territory until it is empty.

\begin{thm}\label{twocopmove}
 For every planar graph $G$ we have $c_2(G) \leq 3$.
\end{thm}	

\begin{proof}
Throughout the proof, we will assume that $H$ is a subgraph of $G$ containing robber's territory and all vertices in the cop territory adjacent to the robber's territory. Initially, $H=G$ and all three cops are at the single vertex $v$. During the game,
we will distinguish three different situations the game can be in:
\begin{enumerate}
	\item[a)] A cop is guarding a nontrivial isometric and bypath-free path $P$ of a subgraph $H$ of $G$, and every path from the robber to the cop territory includes a vertex of $P$. Since $P$ is isometric and bypath-free, the cop can use the leisurely guarding strategy (Lemma \ref{lax}).
	\item[b)] Two cops are guarding $P \cup Q$, where $P$ and $Q$ are internally disjoint isometric and bypath-free paths joining the same two vertices, and any path from the robber to the cop territory includes a vertex of $P \cup Q$. The subgraph $H$ is either in the internal or external region bounded by $P\cup Q$ (we may assume it is in the internal region), including the boundary $P\cup Q$. Two cops are using the leisurely guarding strategy on $P$ and $Q$, respectively.
	\item[c)] A cop is at a vertex $v$ of a subgraph $H$ of $G$, and every path from the robber to the cop territory goes through $v$.
\end{enumerate}

When we say that a cop guards an isometric and bypath-free path in $H$, we will assume that the cop is using the 1-leisurely-guarding strategy given by Lemma \ref{lax}.

Let $C_1$, $C_2$ and $C_3$ be the cops and begin the game by placing them on the same vertex $v$. Let $r$ be the robber's position. If $v$ dominates $G$ we are done. Otherwise, Lemma \ref{bypath-free} guarantees that, by taking $H=G$, and moving at most two cops, we are in case $(a)$ or $(b)$: If we get property $(ii)$ in Lemma \ref{bypath-free}, then by moving one cop to leisurely guard $P$ we obtain case $(a)$. If instead we have outcome $(iii)$, then by moving two cops to $v$ and $u$, we get case $(b)$.

We will show that starting with case $(a)$, $(b)$ or $(c)$, we can move the cops and get again to one of the three cases, increasing the cop territory, and moving at most two cops every turn. We will use $T$ to denote the set of vertices in the cop territory.

Suppose we start with case $(a)$ and $C_1$ is leisurely-guarding $P = v_1v_2 \dots v_k$ with $k \geq 2$ in $H$. Let $Y$ be the component of $H-T$ containing the robber. Then $Y$ contains robber's territory. If there is a unique vertex of $P$ with neighbours in $Y$, say $x$,  then we can move $C_2$ to $x$, getting to case $(c)$.

If $P$ has more than one vertex with neighbours in $Y$, let $v_i$ and $v_j$, with $i < j$, be the first and last vertex of $P$ with neighbours in $Y$, respectively. Let $Q$ be a shortest $(v_i,v_j)$-path, whose internal vertices are in $Y$, and move $C_2$ to guard $Q$. The robber cannot enter the cop territory unless he enters $v_iPv_j$, so we may assume that $P$ is a $(v_i,v_j)$-path since $v_iPv_j$ is also isometric and bypath free in $H$. Notice that the robber is in a component of $Y-Q$ which is either in the bounded or unbounded face determined by $P \cup Q$. Without loss of generality, we may assume he is in the bounded face. If $Q$ can be leisurely guarded by $C_2$ we get case $(b)$, adding the internal vertices of $Q$ to the cop territory. Otherwise, $P$ is isometric and bypath-free in $G[R(P,Q) \cup  V(P)]$ and $Q$ is isometric in $H_Q = G[R(P,Q) \cup V(Q)]$, so we can apply Lemma \ref{movepath} and obtain a bypath $B$ in $H_Q$ as per conclusions of the lemma.

Since $P$ is being leisurely guarded by $C_1$, we can use the turns in which $C_1$ stays still to move $C_3$ and capture the robber's wide shadow on $Q_{\left\langle B \right\rangle}$. Once $C_3$ catches the robber's wide shadow on $Q_{\left\langle B \right\rangle}$, we have two possible outcomes:

\begin{enumerate}
	\item If $r \in  G[R(Q,Q_{\left\langle B \right\rangle}]$, then we get to case $(b)$ by possibly taking subpaths of $Q$ and $Q_{\left\langle B \right\rangle}$, since both paths would be leisurely guarded, and $C_1$ is free to move.
	\item If $r \in  G[R(P,Q_{\left\langle B \right\rangle}]$, cop $C_2$ is free to move and we consider $C_1$ and $C_3$ guarding $P$ and $Q_{\left\langle B \right\rangle}$, respectively. Also, we can apply Lemma \ref{movepath} again until the robber's territory is empty or he is in a region bounded by two leisurely guarded paths, arriving to case $(b)$.
\end{enumerate}

In any case, we have increased the cop territory at least with the bypath $B$.

Now, suppose we start with case $(b)$, with $C_1$ and $C_2$ leisurely-guarding $P$ and $Q$, respectively. Let $Y$ be the component of $H-\left(P \cup Q\right)$ containing the robber. Without loss of generality, we may assume $Y$ is contained in the interior of $P \cup Q$. If only one vertex $v$ of $Y$ has neighbours in $P \cup Q$, then we can use the turns when $C_1$ and $C_2$ stay still to place $C_3$ on $v$ and get to case $c$.

If $P \cup Q$ has exactly two neighbours in $Y$, say $u$ and $v$, let $S$ be a shortest $uv$-path in $Y$. We can move $C_3$ to guard $S$ during the turns when $C_1$ or $C_2$ stay still. Once $C_3$ is guarding $S$, the robber cannot enter the cop territory since he could only do it using $u$ or $v$, which are being covered by $C_3$. This means $C_1$ and $C_2$ are free.

If $S$ has length at most two, we get to case $(a)$ by placing a cop in the middle vertex of $S$ (or any vertex in the case of length one). If $S$ has length more than two, then we can add two vertices, $x$ and $y$, to the graph and make them adjacent to both $u$ and $v$, and draw them one to the left and one to the right of the graph induced by $V(P) \cup V(Q) \cup Y$ (these additional vertices are included only in order to apply Lemma \ref{movepath}). The paths $uxv$, $S$ and $uyv$ form a theta-graph, and the robber is in a component $Y$ of $G-S$ inside the left or the right bounded region defined by the theta graph. Without loss of generality, we may assume he is the region bounded by $P = uxv$ and $S$. The fact that $S$ has length more than two implies $P$ is isometric and bypath-free in the graph, so by applying Lemma \ref{movepath} and moving at most two cops each turn, we can arrive to one of two situations:

\begin{enumerate}
	\item The robber is trapped between $P$ and a path $Q$ which is being leisurely guarded. Then $Q$ is an isometric and bypath-free path in $G[Y]$.
	\item There are two paths $Q$ and $Q'$, such that the robber is in $G[R(Q,Q') \cup V\left(Q\right) \cup V(Q')]$, $Q$ is isometric and bypath-free in $G[R(Q,Q') \cup V(Q)]$, and $Q'$ is isometric and bypath-free in $G[R(Q,Q') \cup V(Q')]$, and both $Q$ and $Q'$ are being leisurely guarded.
\end{enumerate}

Since we get to case $(a)$ in the first situation, and case $(b)$ in the latter, and we add the vertices of $Q\cap Y$ or $Q'\cap Y$ to the cop territory, we make progress.

We may now assume that at least one of $P$ and $Q$ has at least two neighbours in $Y$. Without loss of generality, we may assume $P = v_1v_2 \dots v_k$, with $k \geq 2$, has two distinct neighbours in $Y$. Let $v_i$ and $v_j$ ($i < j$) be the first and last vertex of $P$ with neighbours in $Y$, and let $u_1$ and $u_2$ be their neighbours in $Y$, respectively. Let $S$ be a shortest $(u_1,u_2)$-path in $Y$.

In this case, we can move $C_3$ during the turns when $C_1$ or $C_2$ can stay still and capture the robber's wide shadow in $uPv_iSv_jv$. At this point, the robber will be in one of the bounded regions in the interior of the theta graph formed by $P,Q$ and $S$. Without loss of generality, he is in $R(P,S)$. Since we can leisurely-guard $P$ and we are guarding $S$, cop $C_2$ is free. By applying Lemma \ref{movepath}, we can arrive in a situation where the robber will be trapped in a region bounded by two paths, each of which is leisurely guarded. Again, we have increased the cop territory, and we get to case $(b)$.

Finally, assume we are in case $(c)$. By applying Lemma \ref{bypath-free} it is easy to see that properties $(i)$--$(iii)$ get us to cases $(a)$--$(c)$, respectively, by moving at most two cops in each turn. As before, this step shrinks robber's territory.
\end{proof}

\bibliography{HellyShadowBIB}
\bibliographystyle{plain}

\end{document}